\documentclass[10pt,a4paper]{article}
\usepackage{bbding}
\usepackage[ansinew]{inputenc}
\usepackage{paralist}
\usepackage{comment}
\usepackage{color}
\usepackage{amsmath,amssymb,amsthm}

\newcommand{\PG}{{\rm PG}}

\newcommand{\LL}{{\cal L}}
\newcommand{\WW}{{\cal W}}

\newtheorem{Th}{Theorem}[section]
\newtheorem{Le}[Th]{Lemma}
\newtheorem{Co}[Th]{Corollary}

\newtheorem{Ex}[Th]{Example}
\newtheorem{remarks}[Th]{Remarks}
\newtheorem{remark}[Th]{Remark}

\newcommand{\gauss}[2]{{#1\brack #2}}

\title{Intriguing sets in distance regular graphs}

\author{Stefaan De Winter\thanks{ Department of Mathematical Sciences, Michigan Technological University, MI 49931, USA}\and Klaus Metsch\thanks{Justus-Liebig-Universit\"{a}t, Mathematisches Institut, Arndtstra{\ss}e 2, D-35392 Gie{\ss}en}
}

\date{\today}

\begin{document}

\thispagestyle{empty}
\parindent0ex
\parskip2ex

\maketitle

\begin{abstract}
We construct an infinite family of intriguing sets that are not tight in the Grassmann Graph of planes of $\PG(n,q)$, $n\ge 5$ odd, and show that the members of the family are the smallest possible examples if $n\ge 9$ or $q\ge 25$.
\end{abstract}



\section{Introduction}

Let $\Gamma=(X,E)$ be a connected regular graph of valency $k\ge 1$. Then $k$ is an eigenvalue of $\Gamma$, which will be called its \textsl{trivial eigenvalue}. Following De Bruyn and Suzuki \cite{DeBruyn} we call a set $Y$ of vertices of $\Gamma$ \textsl{ intriguing} if there are integers $x$ and $x'$ such that every vertex of $Y$ is adjacent to $x'$ vertices of $Y$ and every vertex of $X\setminus Y$ is adjacent to $x$ vertices of $Y$. If in addition $Y\not=\emptyset,X$, then $Y$ is called a non-trivial intriguing set. It can be easily seen that a set $Y$ with $Y\not=\emptyset,X$ is intriguing if and only if its characteristic vector lies in the span of the all-one vector and an eigenvector $\lambda$ of $\Gamma$, see for example Proposition 3.3 in \cite{DeBruyn}. In this case $x'-x=\lambda$. If $\lambda$ is the largest or smallest non-trivial eigenvalue of $\Gamma$, then $Y$ is called a \textsl{tight} set.

This definition of intriguing set is the culmination of almost three decades of research on a variety of topics, at first sometimes seemingly unrelated, that share, when formulated well, the property of a set $X$ and $Y$ as above. Included objects are Cameron-Liebler-line classes, introduced by Cameron and Liebler in 1982; tight sets of generalized quadrangles, introduced by Payne in 1987; $m$-ovoids of generalized quadrangles, introduced by Thas in 1989. The first unification of these objects was started by Bamberg, Law and Penttila \cite{Bamberg_tight_quadrangle} in 2009 (but the paper was submitted in 2004) where they introduce the concept of an intriguing set in a generalized quadrangle. Shortly thereafter  Bamberg, Kelly, Law and Penttila further generalized this from generalized quadrangles to polar spaces \cite{Bamberg_tight_polar}. Finally De Bruyn and Suzuki \cite{DeBruyn} provided the generalization to intriguing sets of graphs. Tight sets and intriguing sets have received a lot of attention in recent years; see for example \cite{weighted_intriguing, Cossidente_intriguing, DeBeule_Metsch_tight, Nakic_Storme}. 

We want to make the reader aware of the fact that the terminology of tight set in the previous paragraph, as used in \cite{Bamberg_tight_quadrangle,Bamberg_tight_polar}, does not exactly coincide with the definition from \cite{DeBruyn} which we will be using in this paper. See also page two of \cite{DeBruyn} for a comment on how the definition of tight set has changed from \cite{Bamberg_tight_quadrangle,Bamberg_tight_polar} to \cite{DeBruyn}.

Many examples of tight set are known, but the paper \cite{DeBruyn} does not mention any infinite family of non-trivial intriguing sets that are not tight. In this paper we will construct such a family in the Grassmann graph of planes of $\PG(n,q)$, $n\ge 5$ odd, which is a distance regular graph of diameter $3$. We also show for all odd $n\ge 9$ that our example is the unique smallest example in the Grassmann graph. For $n=7$ we conjecture that this is also true, but we are able to show this only if the order $q$ of $\PG(q,7)$ is at least $13$. For $n=5$, the Grassmann graph of planes of $\PG(n,q)$ has a smaller intriguing set that is not tight, which consists of the planes of an embedded symplectic polar space. We show for $n=5$ and $q\ge 25$ that there do not exist smaller examples, however, we are not able to exclude other intriguing sets of the same cardinality.

\begin{Th}\label{Main}
Let $\LL$ be an intriguing set in the Grassmann graph of planes of $\PG(n,q)$, $n\ge 5$, and suppose that $\LL$ is not tight. Then the following results hold.
\begin{enumerate}
\renewcommand{\labelenumi}{\rm(\alph{enumi})}
\item If $n\ge 9$ is odd, then $$|\LL|\ge \frac {(q^{n+1}-1)(q^{n-1}-1)}{(q-1)(q^2-1)}$$ with equality if and only if $\LL$ consists of all planes of $\PG(n,q)$ that contain a line of a given line spread of $\PG(n,q)$. The same results holds for $n=7$ if $q\ge13$.
\item If $n=5$ and $q\ge 25$, then $|\LL|\ge (q+1)(q^2+1)(q^3+1)$ and equality can be achieved by the set of planes of a symplectic polar space $W(5,q)$ embedded in $\PG(5,q)$.
\end{enumerate}
\end{Th}

For $n\ge 6$ even, we can prove better bounds but we do not have examples meeting the bound. In fact, we do not know of any intriguing set of the Grassmann graph of planes of $\PG(n,q)$, $n$ even, that is not tight.

\begin{Th}
Let $\LL$ be an intriguing set in the Grassmann graph of planes of $\PG(n,q)$ and suppose that $\LL$ is not tight. If $n=6$, then $|\LL|\ge (q+1)^2(q^3-1)/(q-1)$, and if $n\ge 8$, then $\LL\ge q(q^{n-2}-1)/(q^2-1)$.
\end{Th}

Throughout the paper we will use the following notation. For integers $n\ge k\ge 0$ the Gaussian coefficients are
\[
\gauss{n}{k}_q:=\prod_{i=1}^k\frac{q^{n+1-i}-1}{q^i-1}.
\]
Usually we will omit the subscript $q$ and simply write $\gauss{n}{k}$. We also define $\theta_n:=\gauss{n+1}{1}$ for all integers $n\ge 1$.

The paper is structured as follows: in Section 2 we provide examples of intriguing sets that are of interest to our results, in Section 3 we prove Theorem \ref{Main} for $n\ge 9$ as well as Theorem 1.2, in Section 4 we prove Theorem \ref{Main} for $n=5$, and finally in Section 5 we present the proof of Theorem \ref{Main} for $n=7$.

\section{Examples}

First we give examples of intriguing sets in the Grassmann graph of planes of $\PG(n,q)$, $n\ge 5$. This is a distance regular graph of diameter three and its three non-trivial eigenvalues are (Theorem 9.3.3 in \cite{Brouwer})
\[
\lambda_1=\frac{q^2(q+1)(q^{n-3}-1)}{q-1},\  \lambda_2=\frac{q^{n-1}-q^3}{q-1}-q-1,\ \lambda_3=-(q^2+q+1).
\]
Recall that a set $\LL$ of planes of $\PG(n,q)$ is an intriguing set of the Grassmann graph, if every plane of $\LL$ meets the same number $x'$ of planes of $\LL$ in a line and every plane not in $\LL$ meets the same number $x$ of planes of $\LL$ in a line. In this case $x'-x$ is one of the eigenvalues $\lambda_1$, $\lambda_2$ and $\lambda_3$ \cite[Proposition 3.7]{DeBruyn}. Thus $x'=x+\lambda_i$ for some $i\in\{1,2,3\}$. Counting pairs $(\pi,\tau)$ of planes that intersect in a line and such that $\tau\in\LL$ gives $|\LL|x'+(\gauss{n}{3}-|\LL|)x=|\LL|\theta_2(\theta_{n-2}-1)$, so that $|\LL|$ is uniquely determined by $x$ and the eigenvalue $\lambda_i=x'-x$. For the eigenvalue $\lambda_2$ we have
\[
|\LL|=\frac{(q^{n+1}-1)(q^{n-1}-1)}{(q+1)(q^2-1)(q^3-1)}\cdot x\ .
\]

\begin{Ex}\label{Example_point}
If $P$ is a point of $\PG(n,q)$, then the set consisting of all planes of $\PG(n,q)$ on $P$, $n\ge 5$, is a tight set with $x=\theta_2$ and $x'-x=\lambda_1$.
\end{Ex}

\begin{Ex}\label{Example_hyperplane}
If $H$ is a hyperplane of $\PG(n,q)$, then the set consisting of all planes of $H$ is a tight set of $\PG(n,q)$ with $x=\theta_{n-3}$ and $x'-x=\lambda_1$.
\end{Ex}

\begin{Ex}\label{ex2}
Consider a line spread $S$ of $\PG(n,q)$, $n\ge 5$ odd. The set $\LL$ of all planes of $\PG(n,q)$ that contain a line of $S$ is an intriguing set with $x=(q+1)(q^2+q+1)$ and $x'-x=\lambda_2$.
\end{Ex}

\begin{Ex}\label{ex1}
Consider the symplectic polar space consisting of the totally isotropic subspaces of a symplectic polarity of $\PG(5,q)$, and let $\LL$ be the set of planes of $W(5,q)$. This is an intriguing set of the Grassmann graph of planes of $\PG(5,q)$ with $x:=(q+1)^2$ and $x'-x=\lambda_2=q^3-q-1$. Also \[
|\LL|=(q^3+1)(q^2+1)(q+1)=x(q+1)(q^2-q+1).
\]
\end{Ex}

To end this section, we give an example of intriguing sets that are not tight in a dual polar graph.

\begin{Ex}
Let $\Gamma$ be the dual polar graph of $W(5,q)$, that is the vertices are the planes of $W(5,q)$ with two vertices adjacent if the corresponding planes meet in a line. Suppose that $S$ is a line spread of $W(5,q)$, that is, a set of lines of $W(5,q)$ covering every point exactly once. The set $\LL$ of all planes of $W(5,q)$ that contain a line of $S$ is an intriguing set with $x:=q^2+q+1$ and $x'-x=-1$. Also
\[
|\LL|=(q^4+q^2+1)(q+1)=x(q+1)(q^2-q+1).
\]
The graph $\Gamma$ is distance regular of diameter three and $-1$ is its second smallest eigenvalue. Hence $\LL$ is not a tight set.
\end{Ex}

\begin{remarks}
\begin{enumerate}
\renewcommand{\labelenumi}{\rm(\alph{enumi})}
\item We do not know in general of the existence of a line spread of $W(5,q)$. Hering \cite{Hering} constructed two line spreads of $W(5,3)$; see also Buekenhout \cite{Buekenhout}.  John Bamberg informed us that $W(5,q)$ in fact has over a thousand projectively inequivalent line spreads for $q=3$ (computer search). However we do not know of any example for other values of $q$. 
\item Notice that a plane spread of $W(5,q)$ is also an intriguing set of the dual polar graph of $W(5,q)$, but this corresponds to the smallest eigenvalue, so a spread is tight.
\end{enumerate}
\end{remarks}

\section{General remarks on intriguing sets in the Grassmannian of planes}

Let $\LL$ be an intriguing set in the Grassmann graph of planes of $\PG(n,q)$ corresponding to the eigenvalue
\[
\lambda_2:=\theta_{n-5}q^3-q-1=\theta_{n-2}-1-(q+1)^2.
\]
Thus for some integer $x>0$ we have that each plane that is not in $\LL$ meets $x$ planes of $\LL$ in a line and each plane of $\LL$ meets $x':=\lambda_2+x$ planes of $\LL$ in a line. If $n$ is odd, then an example with $x=(q+1)\theta_2$ is given in Example \ref{ex2}. We will show for odd $n\ge 9$, that this is the smallest possible value for $x$. For $n=7$ we will be able to prove this provided $q\ge13$. For $n=5$ we have seen that $x=(q+1)^2$ is possible (Example \ref{ex1}) and we shall also see that this is optimal. The cases $q=5$ and $7$ will be dealt with in the next two sections.

\begin{Le}\label{incidenceswithpointsandhyperplanes}
\begin{enumerate}
\renewcommand{\labelenumi}{\rm(\alph{enumi})}
\item We have
\begin{eqnarray}\label{eqn_L_general}
|\LL|=\frac{\theta_n\theta_{n-2}}{(q+1)^2\theta_2}x,
\end{eqnarray}
\item Every point is contained in $|\LL|\theta_2/\theta_n$ planes of $\LL$.
\item Every hyperplane contains $|\LL|\theta_{n-3}/\theta_n$ planes of $\LL$.
\end{enumerate}
\end{Le}
\begin{proof}
(a) This we have seen in Section 2.

(b) Consider the set $T$ of planes on a point. We have seen in Example \ref{Example_point} that $T$ is a tight set for the eigenvalue $\lambda_1$ of the Grassmann graph of planes of $\PG(n,q)$. By hypothesis $\LL$ is an intriguing set of the same graph for the eigenvalue $\lambda_2$. Therefore it follows from Corollary 3.6 in in \cite{DeBruyn} that $|\LL\cap T|=|T|\cdot|\LL|/\gauss{n+1}{3}$.

(c) Consider the set $T$ of planes in a hyperplane. This is a tight set for the eingenvalue $\lambda_1$ (Example \ref{Example_hyperplane}), so it follows from Corollary 3.6 in \cite{DeBruyn} that $|\LL\cap T|=|T|\cdot|\LL|/\gauss{n+1}{3}$.
\end{proof}

\begin{Co}\label{divisibility}
\begin{enumerate}
\renewcommand{\labelenumi}{\rm(\alph{enumi})}
\item If $3\mid n$, then $\theta_2\mid x$.
\item If $n$ is even, then $(q+1)^2\mid x$.
\item If $n$ is odd, then $(q+1)^2\mid\theta_{n-2}x$.
\end{enumerate}
\end{Co}

\begin{Le}
If some plane of $\LL$ has the property that each of its lines lies in a plane that is not contained in $\LL$, then
$x-1\ge (\lambda_2+1)/(q^2+q)$.
\end{Le}
\begin{proof}
Let $\pi$ be a plane of $\LL$. Then $\pi$ meets $\lambda_2+x$ other planes of $\LL$ in a line, so some line $\ell$ of $\pi$ lies in at least $(\lambda_2+x)/\theta_2$ further planes of $\LL$. By hypothesis, $\ell$ lies in a plane $\tau$ that does not belong to $\LL$. Then $\tau$ meets exactly $x$ planes of $\LL$ in a line, so it follows that $(\lambda_2+x)/\theta_2\le x-1$.
\end{proof}

\begin{Le}\label{linespread}
Suppose that every plane of $\LL$ has a line with the property that every plane on that line belongs to $\LL$. Suppose that $x<q\theta_{n-3}+(q+1)^2$. Then $n$ is odd, $x=(q+1)\theta_2$ and $\LL$ is based on a line spread of $\PG(n,q)$, that is $\LL$ consists of all planes containing a line of a given line spread.
\end{Le}
\begin{proof}
In this proof we call a line of $\PG(n,q)$ a full line, if all $\theta_{n-2}$ planes on the line belong to $\LL$. By hypothesis, every plane of $\LL$ contains a full line.

Assume some plane $\pi$ contains two full lines $\ell_1$ and $\ell_2$. Then $\pi$ meets at least $2(\theta_{n-2}-1)$ other planes of $\LL$ in a line. Hence $2(\theta_{n-2}-1)\le \lambda_2+x$. This is equivalent to $x\ge q\theta_{n-3}+(q+1)^2$. This contradicts the hypotheses of the lemma.

Hence every plane of $\pi$ has a unique full line. Then the full lines form a partial spread $F$ and $\LL$ consists of the $|F|\theta_{n-2}$ planes containing a line of $F$. Consider a line $\ell\in F$. If a plane $\pi$ on $\ell$ meets exactly $c$ lines of $F\setminus\{\ell\}$, then $\pi$ meets exactly $\theta_{n-2}-1+c(q+1)$ planes of $\LL$ in a line, since each of the $c$ lines lies on $q+1$ planes that meet $\pi$ in a line. This shows that $c$ is independent of $\pi$. Since the lines in $F\setminus\{\ell\}$ cover $(|F|-1)(q+1)$ points and $\ell$ lies on $\theta_{n-2}$ planes, we see that $c=(|F|-1)(q+1)/\theta_{n-2}$. As each plane of $\LL$ meets exactly $\lambda_2+x$ planes of $\LL$ in a line, it follows that $\lambda_2+x=\theta_{n-2}-1+c(q+1)$. Since $\lambda_2=\theta_{n-2}-1-(q+1)^2$, it follows that
\begin{eqnarray}\label{eqn_wow}
x-(q+1)^2=\frac{(|F|-1)(q+1)^2}{\theta_{n-2}}.
\end{eqnarray}
Combining $|\LL|=|F|\theta_{n-2}$ and (\ref{eqn_L_general}) gives $x=|F|(q+1)^2\theta_2/\theta_n$. Then (\ref{eqn_wow}) implies that $|F|=\theta_n/(q+1)$ and $x=(q+1)\theta_2$. Hence $n$ is odd, $F$ is a spread of $\PG(n,q)$, and $\LL$ is based on this spread.
\end{proof}

\begin{Co}
If $n\ge 8$ is even, then $x\ge q\theta_{n-4}/(q+1)$. If $n=6$, then $x\ge (q+1)^2\theta_2$.
\end{Co}
\begin{proof}
The first statement follows from the fact that when $n$ is even, $\LL$ cannot be based on a line spread, and hence by the two previous lemmas $x\ge \min((\lambda_2+1)/(q^2+q)+1, q\theta_{n-3}+(q+1)^2+1)=q\theta_{n-4}/(q+1)$. The second statement follows from Corollary \ref{divisibility}.
\end{proof}

\begin{Th}
If $n\ge 9$ is odd, then $x\ge (q+1)\theta_2$ and equality implies that $\LL$ is based on a line spread.
\end{Th}

For $n=7$ the above arguments only show $x\ge q^3+q$. It will take considerably more effort in Section \ref{Sectionnis7} to show that the theorem also holds for $n=7$ and we can show this only for $q\ge 13$. For $n=5$, the theorem does not remain true, since the set of all planes of a symplectic polar space $W(5,q)$ embedded in $\PG(5,q)$ provides an intriguing set for the eigenvalue $\lambda_2$ with $x=(q+1)^2$. This case is handled in Section \ref{Sectionnis5}.

\section{The case $n=5$ and planes of $W(5,q)$}\label{Sectionnis5}

In this section $\LL$ denotes a non-trivial intriguing set of planes of $PG(5,q)$ corresponding to the second non-trivial eigenvalue $\lambda_2:=q^3-q-1$ of the Grassmann graph. Hence there  exists an integer $x>0$ such that $|\LL|=x(q^2+1)(q^2-q+1)$, every plane of $\LL$ meets $\lambda_2+x$ planes of $\LL$ in a line, and every plane not in $\LL$ meets $x$ planes of $\LL$ in a line. Examples are given in \ref{ex1} and \ref{ex2}. In the example of \ref{ex1} we have $x=(q+1)^2$ and we will show in this section that $x$ can not be smaller. More precisely we will prove the following theorem.

\begin{Th}\label{THW5q}
Let $\LL$ be a set of planes of $\PG(5,q)$ with $q\ge 25$ such that $\LL$ is not empty and does not contain all planes of $\PG(5,q)$. Suppose that there exists an integer $x>0$ such that
\[
|\LL|=x(q^2+1)(q^2-q+1)
\]
every plane of $\LL$ meets $\lambda_2+x$ planes of $\LL$ in a line, and every plane not in $\LL$ meets $x$ planes of $\LL$ in a line. Then $x\ge (q+1)^2$. Also, if $x=(q+1)^2$, then one of the following holds.\begin{enumerate}
\renewcommand{\labelenumi}{\rm(\alph{enumi})}
\item There exists a solid $S$ of $\PG(5,q)$ such that every plane of $S$ belongs to $\LL$, and every plane of $\LL$ is contained in $S$ or meets $S$ in exactly one point, or, dually, there exists a line $\ell$ of $\PG(5,q)$ such that all planes on $\ell$ belong to $\LL$, and every plane of $\LL$ contains $\ell$ or is skew to $\ell$.
\item Every line and every solid is incident with either $0$ or $q+1$ planes of $\LL$.
\end{enumerate}
\end{Th}

\begin{remark}
We do not know whether (1) can happen. The planes of an embedded symplectic polar space provide an example with $x=(q+1)^2$ that satisfies (2).
\end{remark}

We prove this theorem in a series of lemmas. Throughout we assume that $\LL$ is a set of planes of $\PG(5,q)$ as in the theorem and that
\begin{align}\label{n5_bound_on_x}
x\le (q+1)^2.
\end{align}

\begin{Le}\label{incidenceswithpointsandhyperplanes_special_n5}
Every point and every hyperplane is incident with exactly $(q^2+1)x/(q+1)$ planes of $\LL$. Hence $q+1$ divides $2x$.
\end{Le}
\begin{proof}
This is a special case of Lemma \ref{incidenceswithpointsandhyperplanes} (b) and (c).
\end{proof}

\begin{Le}\label{z1z2}
Let $S$ be a solid, $\beta$ the number of planes of $\LL$ in $S$ and $z$ the number of planes of $S$ meeting $S$ in a line. Then $z=(q^2+1)x-\beta(q+1)$.
\end{Le}
\begin{proof}
By Lemma \ref{incidenceswithpointsandhyperplanes_special_n5},    each of the $q+1$ hyperplanes on $S$ contains $(q^2+1)x/(q+1)$ planes of $\LL$ and all these planes lie in $S$ or meet $S$ in a line. Conversely, the $\beta$ planes of $\LL$ contained in $S$ lie in all of these $q+1$ hyperplanes, whereas the $z$ planes of $\LL$ that meet $S$ in a line, lie in exactly one of these. Hence $(q^2+1)x=(q+1)\beta+z$.
\end{proof}

\begin{Le}\label{atmostx}
\begin{enumerate}
\renewcommand{\labelenumi}{\rm(\alph{enumi})}
\item
Suppose that there exists a solid $S$ that is incident with more than $x$ planes of $\LL$. Then $x=(q+1)^2$ and every plane of  $S$ is contained in $\LL$. Also every plane of $\LL$ that is not incident with $S$ meets $S$ in a point.
\item Suppose that there exists a line $\ell$ that is incident with more than $x$ planes of $\LL$. Then $x=(q+1)^2$, every plane that contains $\ell$ lies in $\LL$. Also, every plane of $\LL$ that does not contain $\ell$ is skew to $\ell$.
\end{enumerate}
\end{Le}
\begin{proof}
Since (a) and (b) are dual statements, it suffices to prove (a). Suppose therefore that there exists a solid $S$ such that more than $x$ planes of $\LL$ are contained in $S$. As every plane that is not in $\LL$ meets exactly $x$ planes of $\LL$ in a line, it follows that all planes of $S$ belong to $\LL$. Since all $\theta_3$ planes of $S$ belong to $\LL$, Lemma \ref{z1z2} and (\ref{n5_bound_on_x}) imply that $x=(q+1)^2$ and that there does not exist a plane in $\LL$ that meets $S$ in a line.
\end{proof}

The statement of the previous lemma describes the situation of (a) in Theorem \ref{THW5q}. In order to prove Theorem \ref{THW5q} we can therefore assume from now on that every line and every solid is incident with at most $x$ planes of $\LL$. By $\beta$ we denote the largest integer such that there exists a solid or a line that is incident with exactly $\beta$ planes of $\LL$. Then
\begin{align}\label{inequality_beta}
2\le \beta\le x.
\end{align}
Here $\beta\ge 2$ holds, since every plane of $\LL$ intersects $\lambda_2+x>0$ planes of $\LL$ in a line. We remark that the properties of $\LL$ are selfdual, that is, $\LL$ as a set of planes of the dual projective space has the same properties. Hence, by switching to the dual space, if necessary, it would be no loss of generality to assume that there exists a solid with $\beta$ planes of $\LL$. We will do so later.

\begin{Le}\label{basiccountingargument}
Suppose that $S$ is a solid that contains $\beta$ planes of $\LL$. Let $\alpha$ be an integer with $2\le \alpha\le \beta$, and let $\pi_1,\dots,\pi_\alpha$ be planes of $\LL$ that are contained in $S$. Let $G$ be the set of lines of $S$ that are incident with at least two planes of $\{\pi_1,\dots,\pi_\alpha\}$. For $l\in G$ let $w_l$ be the number of planes of $\{\pi_1,\dots,\pi_\alpha\}$ that contain $l$ and let $g_l$ be the number of planes $\pi$ of $\LL$ with $\pi\cap S=l$.
\begin{enumerate}
\renewcommand{\labelenumi}{\rm(\alph{enumi})}
\item $\sum_{l\in G}w_l(w_l-1)=\alpha(\alpha-1)$ and $\sum_{l\in G}(w_l-1)\le \frac12\alpha(\alpha-1)$.
\item We have
\begin{eqnarray}\label{eqn_basiccountingargument}
\alpha(x+\lambda_2+1-\beta)\le (q^2+1)x-\beta(q+1)+\sum_{l\in G}(w_l-1)g_l
\end{eqnarray}
and equality implies that every line of $S$ that lies in a plane of $\LL$ is a line of one of the planes $\pi_1,\dots,\pi_\alpha$.
\end{enumerate}
\end{Le}
\begin{proof}
(a) Count triples $(\pi_i,\pi_j,l)$ with $i\not=j$ and lines $l$ satisfying $l=\pi_i\cap\pi_j$ to find $\sum_{l\in G} w_l(w_l-1)=\alpha(\alpha-1)$. As $w_l\ge 2$ for $l\in G$, it follows that $\sum_{l\in G}(w_l-1)\le \frac12\alpha(\alpha-1)$.

(b) Let $X$ be the set of all lines that lie in at least one of the planes $\pi_i$ and define $w_l$ and $g_l$ for $l\in X$ as they are defined for $l\in G$. Then $G\subseteq X$. We count triples $(i,l,\tau)\in\{1,\dots,\alpha\}\times X\times\LL$ such that $\tau\cap S=l\subseteq \pi_i$. Each plane $\pi_i$ meets $x+\lambda_2$ planes of $\LL$ in a line and exactly $\beta-1$ of these lie in $S$. Hence, each $i$ occurs in exactly $x+\lambda_2+1-\beta$ such triples. Each line of $X$ occurs in exactly $w_lg_l$ such triples. Hence
\[
\alpha(x+\lambda_2+1-\beta)=\sum_{l\in X}w_lg_l=\sum_{l\in X}g_l+\sum_{l\in G}(w_l-1)g_l.
\]
Since $\sum_{l\in X}g_l$ is the number of planes $\LL$ that meet $S$ in a line of $X$, the statement now follows from Lemma \ref{z1z2}.
\end{proof}

\begin{Le}\label{atleastqplus1}
\begin{enumerate}
\renewcommand{\labelenumi}{\rm(\alph{enumi})}
\item There exists a solid that contains at least $q+1$ planes of $\LL$.
\item If every solid and every line is incident with at most $q+1$ planes of $\LL$, then $x=(q+1)^2$ and every line and every solid is incident with either no or exactly $q+1$ planes of $\LL$.
\end{enumerate}
\end{Le}
\begin{proof}
For each solid $S$ denote by $t_S$ the number of planes of $\LL$ contained in $S$. As every plane of $\LL$ is incident with $\theta_2$ solids, then $\sum_St_S=|\LL|\theta_2$. Counting triples $(\pi,\pi',S)$ of different planes $\pi,\pi'\in\LL$ such that $\pi\cap\pi'$ is a line and $S$ is the solid spanned by $\pi$ and $\pi'$ gives $\sum_St_S(t_S-1)=|\LL|(\lambda_2+x)$, since every plane of $\LL$ meets $\lambda_2+x$ planes of $\LL$ in a line. Put $c:=\max\{t_S\mid \mbox{$S$ is a solid}\}$. Then
\begin{eqnarray}
0&\le&\sum_St_S(c-t_S)=\sum_St_S(c-1)-\sum_St_S(t_S-1)\nonumber
\\
&=&(c-1)|\LL|\theta_2-|\LL|(x+\lambda_2)\label{eqn0orqplus1}
\\
&=&|\LL|((c-1)\theta_2-x-\lambda_2).\nonumber
\end{eqnarray}
It follows that $x+\lambda_2\le (c-1)\theta_2$. As $\lambda_2=q^3-q-1$, this implies that $c>q-1$. But $c=q$ leads to $x\le (c-1)\theta_2-\lambda_2=q$ and hence to $c=q=x$. As $q+1$ divides $2x$ (Lemma \ref{incidenceswithpointsandhyperplanes_special_n5}), this is impossible. Hence $c\ge q+1$. This proves (a).

For the proof of (b) suppose now that every line and every solid is incident with at most $q+1$ planes of $\LL$. Then $x+\lambda_2\le (c-1)\theta_2$ gives $x\le (q+1)^2$. If $x=(q+1)^2$, then (\ref{eqn0orqplus1}) shows that every solid contains either $0$ or $q+1$ planes of $\LL$. Dually, in this case, every line is incident with no or $q+1$ planes of $\LL$.

Assume that $x<(q+1)^2$. As $q+1$ divides $2x$ (Lemma \ref{incidenceswithpointsandhyperplanes_special_n5}), it follows that $x\le(q+1)^2-\frac12(q+1)$. To derive a contradiction we consider a solid $S$ that contains exactly $c=q+1$ planes $\pi_0,\dots,\pi_q$ of $\LL$. Lemma \ref{basiccountingargument} applied to $S$ with $\alpha=q+1$ gives
\[
(q+1)(x+\lambda_2-q)\le (q^2+1)x-(q+1)^2+{q+1\choose 2}(q-1),
\]
since in this lemma $g_h\le q+1-w_h\le q-1$ for all $h\in G$. It follows that $x\ge (q+1)^2-\frac12(q+1)$. Hence, we have equality. Therefore Lemma \ref{basiccountingargument} shows that every plane of $\LL$ that meets $S$ in a line has the property that this line lies in at least one of the planes $\pi_i$. Consider a plane $\tau$ of $S$ that does not belong to $\LL$. Then each plane of $\LL$ that meets $\tau$ in a line meets it in one of the lines $\pi_i\cap\tau$. Since every line lies in at most $q+1$ planes of $\LL$, it follows that $\tau$ meets at most $(q+1)^2$ planes of $\LL$ in a line. Hence $\lambda_2+x\le(q+1)^2$. But $\lambda_2=q^3-q-1$, a contradiction.
\end{proof}

\begin{Le}\label{betasmallorlarge}
Suppose that $x\le (q+1)^2$ and $q\ge 4$. Then $\beta= q+1$ or $\beta\ge \frac12(q^2+5)$.
\end{Le}
\begin{proof}
By passing through the dual space, if necessary, we may assume that some solid $S$ contains $\beta$ planes of $\LL$. We know from the previous lemma that $\beta\ge q+1$. Let $\alpha$ be an integer with $q+1\le \alpha\le \beta$ and let $\pi_1,\dots,\pi_\alpha$ be $\alpha$ different planes of $\LL$ that are contained in $S$. We will apply Lemma \ref{basiccountingargument} to these planes using that every line is contained in at most $\beta$ planes of $\LL$. Hence in Lemma \ref{basiccountingargument} we have $g_l\le \beta-w_l$ for $l\in G$, which implies, using part (a) of Lemma \ref{basiccountingargument}, that
\begin{eqnarray*}
\sum_{h\in G}(w_h-1)g_h\le\sum_{h\in G}(w_h-1)(\beta-w_h)= \frac{1}{2}\alpha(\alpha-1)\beta-\alpha(\alpha-1).
\end{eqnarray*}
Therefore part (b) of Lemma \ref{basiccountingargument} gives
\begin{eqnarray}\label{eqn_basic}
\alpha(x+\lambda_2-\beta+\alpha-\frac12(\alpha-1)\beta)\le (q^2+1)x-(q+1)\beta.
\end{eqnarray}
Case 1. We assume that $\beta< 2q$. In this case we choose $\alpha=\beta$ and find
\[
\beta(x+\lambda_2+q+1-\frac12(\beta-1)\beta)\le (q^2+1)x.
\]
Since $\lambda_2=q^3-q-1$ and $x\le (q+1)^2$, it follows that $f(\beta)\ge 0$ for the polynomial $f$   defined by
\[
f:=(q^2+1)(q+1)^2-x\left(q^3+(q+1)^2-\frac12(x-1)x\right)\in\mathbb{R}[x].
\]
As $q\ge 4$ we have $f(q+2)=-(q^3-q^2-2q-2)/2<0$ and $f'(x)=-q^3-(q+1)^2+\frac{3}{2}x^2-x<0$ for $q+2\le x<2q$. Hence $f(x)<0$ for all $x\in\mathbb{R}$ with $q+2\le x< 2q$. Since $f(\beta)\ge 0$ and since we assume that $\beta< 2q$ in this case, it follows that $\beta\le q+1$, and hence, using Lemma \ref{atleastqplus1}, that $\beta=q+1$.

Case 2. We assume that $\beta\ge 2q$. Then we choose $\alpha=2q$. Since  $\lambda_2=q^3-q-1$ and $x\le (q+1)^2$ we find from (\ref{eqn_basic})
\[
2q(q^3+q^2+3q-\frac12(2q+1)\beta)\le (q^2+1)(q+1)^2-(q+1)\beta.
\]
For $q\ge 4$, this gives $\beta>(q^2+4)/2$. Hence $\beta\ge \frac12(q^2+5)$.
\end{proof}

\begin{Le}\label{gammalower}
Suppose that $q\ge 4$ and that  is a solid that contains $\beta$ planes of $\LL$. There exists a plane $\tau_0\in\LL$ that lies in $S$, three distinct lines $\ell_1$, $\ell_2$ and $\ell_3$ of $\tau_0$, and an integer
\[
\gamma\ge \frac{\beta(q^3-q-\beta-x)}{(q^2+1)x-(q+1)\beta}
\]
such that each line $\ell_1$, $\ell_2$ and $\ell_3$ lies on at least $\gamma$ planes of $\LL$ that are contained in $S$.
\end{Le}
\begin{proof}
Let $T$ be the set of the $\beta$ planes of $S$ that are contained in $\LL$, and let $X$ be the set of all planes $\pi$ of $\LL$ for which $S\cap\pi$ is a line. Lemma \ref{z1z2} shows that $n:=|X|=(q^2+1)x-(q+1)\beta$. Each plane of $T$ meets $x+\lambda_2$ planes of $\LL$ in a line, and exactly $x+\lambda_2+1-\beta$ of these planes are of $X$. Hence there are exactly $\beta(x+\lambda_2+1-\beta)$ pairs $(\tau,\pi)\in T\times X$ such that $\pi\cap \tau$ is a line. Put $X=\{\pi_1,\dots,\pi_n\}$ and let $\gamma_i$ be the number of planes in $T$ that contain the line $S\cap \pi_i$. We may assume that $\gamma_i\ge \gamma_j$ for $i\le j$. We have $\beta(x+\lambda_2+1-\beta)=\sum \gamma_i$.

From (\ref{n5_bound_on_x}) and (\ref{inequality_beta}) and $q\ge 4$  we have $\lambda_2=q^3-q-1>2(q+1)^2\ge x+\beta$. Hence $\sum\gamma_i>2\beta x$ and thus there exists  a smallest positive integer $s$ such that $\sum_{i=1}^s\gamma_i>2\beta x$. Put $\gamma:=\gamma_s$. Then
\begin{eqnarray*}
\beta(x+\lambda_2+1-\beta)=\sum_{i=1}^n\gamma_i\le 2\beta x+(n+1-s)\gamma.
\end{eqnarray*}
As $s\ge 1$ and $\lambda_2=q^3-q-1$, it follows that $\beta(q^3-q-\beta-x)\le n\gamma$, that is $\gamma$ satisfies the inequality in the statement.

Since $\sum_{i=1}^s\gamma_i>2\beta x$, there are more than $2\beta x$ pairs $(\tau,\pi)$ with $\tau\in T$ and $\pi\in\{\pi_1,\dots,\pi_s\}$ such that $\tau\cap\pi$ is a line. Since $|T|=\beta$ it follows that there exists a plane $\tau\in T$ that occurs in more than $2 x$ pairs $(\tau,\pi)$ with $\pi\in\{\pi_1,\dots,\pi_s\}$. Since each line lies in at most $x$ planes of $\LL$, each line of $\tau$ lies in at most $x$ of the planes of $\{\pi_1,\dots,\pi_s\}$. It follows that $\tau$ meets three planes $\pi_i$, $\pi_j$ and $\pi_k$ with $1\le i<j<k\le s$ in distinct lines. Then $\gamma_i,\gamma_j,\gamma_k\ge \gamma_s=\gamma$, so each of these three lines of $\tau$ lies in at least $\gamma$ planes of $T$.
\end{proof}

\begin{Le}\label{gammaupper}
Suppose that $S$ is a solid that contains $\beta$ planes of $\LL$, let $\tau$ be one of these planes, let $\ell_1$, $\ell_2$ and $\ell_3$ be three distinct lines of $\tau$, and suppose that $\gamma\ge 3$ is an integer such that each of these three lines lies in at least  $\gamma$ planes of $\LL$ that are contained in $S$. Then
\[
(3\gamma-2)(x+q^3-q-2\beta)\le (q^2+1)x-\beta(q+1)-7\beta+3\theta_3.
\]
\end{Le}
\begin{proof}
Put $\alpha=3\gamma-2$ and consider $\alpha$ planes $\pi_1,\dots,\pi_\alpha\in\LL$ of $S$ such that $\pi_1=\tau$ and such that each line $\ell_i$ lies in exactly $\gamma$ of these planes. We will apply Lemma \ref{basiccountingargument} and use the notation there. We distinguish two cases in order to give an upper bound for the sum $\sum (w_h-1)g_h$ in (\ref{eqn_basiccountingargument}).

Case 1. The three lines $\ell_i$ contain a common point $P$ of $\tau$. Then $w_{\ell_i}=\gamma$ and $w_h\le 3$ for every other line $h$ of $G$. The sum $\sum_{h\in G}g_h$ is the number of planes of $\LL$ that meet $S$ in a line of $G$. Since all lines of $G$ contain $P$ and since $P$ lies on at most $\theta_3$ planes of $\LL$ (Lemma \ref{incidenceswithpointsandhyperplanes_special_n5}), we see that the sum is at most $\theta_3-\alpha$. Since every line lies in at most $\beta$ planes of $\LL$, we have $g_{\ell_i}\le\beta-\gamma$. Hence
\begin{eqnarray*}
\sum_{h\in G}(w_h-1)g_h
&\le& (\gamma-3)(g_{\ell_1}+g_{\ell_2}+g_{\ell_3})+2\sum_{h\in G}g_h
\\
&\le & 3(\gamma-3)(\beta-\gamma)+2(\theta_3-\alpha).
\end{eqnarray*}

Case 2. The three lines $\ell_i$ form a triangle with three intersection points $P_1,P_2,P_3$ in $\tau$. Then $w_{\ell_i}=\gamma$ and $w_h= 2$ for every other line $h$ of $G$. The sum $\sum_{h\in G}g_h$ is the number of planes of $\LL$ that meet $S$ in a line of $G$. All lines of $G$ contain one of the points $P_1$, $P_2$ and $P_3$. Each of these points $P_i$ lies in at most $\theta_3$ planes of $\LL$ (Lemma \ref{incidenceswithpointsandhyperplanes_special_n5}) and lies on exactly $2\gamma-1$ of the planes $\pi_1,\dots,\pi_\alpha$. It follows that $\sum_{h\in G}g_h$ is at most  $3(\theta_3-(2\gamma-1))-g_{\ell_1}-g_{\ell_2}-g_{\ell_3}$. Since every line lies in at most $\beta$ planes of $\LL$, we have $g_{\ell_1}\le\beta-\gamma$. Hence
\begin{eqnarray*}
\sum_{h\in G}(w_h-1)g_h
&=& (\gamma-2)(g_{\ell_1}+g_{\ell_2}+g_{\ell_3})+\sum_{h\in G}g_h
\\
&\le& 3(\gamma-3)(\beta-\gamma)+3(\theta_3+1-2\gamma).
\end{eqnarray*}
As $\alpha=3\gamma-2$, the bound in Case 2 is weaker than the one in Case 1, so we can continue with the bound in Case 2, which we write in the following form
\[
\sum_{h\in G}(w_h-1)g_h \le  (3\gamma-2)(\beta-\gamma+\frac13)-7\beta+\frac{11}3+3\theta_3.
\]
With this bound we apply (\ref{eqn_basiccountingargument}) of Lemma \ref{basiccountingargument}. As $\alpha=3\gamma-2$ this gives
\[
(3\gamma-2)(x+\lambda_2+1-2\beta+\gamma-\frac13)\le (q^2+1)x-\beta(q+1)-7\beta+\frac{11}3+3\theta_3.
\]
As $\gamma\ge 3$ and $\lambda_2=q^3-q-1$, this proves the assertion of the lemma.
\end{proof}

\begin{Le}\label{betanottoolarge}
If $q\ge 25$, then $\beta<\frac12(q^2+5)$.
\end{Le}
\begin{proof}
We may assume that there exists a solid with $\beta$ planes of $\LL$. Since $\beta\le x\le (q+1)^2$, Lemma \ref{gammalower} gives an integer
\[
\gamma\ge f(\beta):=\frac{\beta(q^3-q-2(q+1)^2)}{(q^2+1)(q+1)^2-(q+1)\beta}
\]
such that $S$ has a plane and three lines in that plane that lie each on at least $\gamma$ planes of $\LL$ that are contained in $S$. Assume that $\beta\ge\frac12(q^2+5)$. Since $q\ge 25$, it follows that
\[
\gamma\ge f(\beta)\ge f(\frac12(q^2-1))=\frac{(q-1)(q^2-3q-2)}{2q^2-q+3}\ge 3.
\]
Lemma \ref{gammaupper} gives the upper bound \[
(3\gamma-2)(x+q^3-q-2\beta)\le (q^2+1)x-\beta(q+1)-7\beta+3\theta_3.
\]
Since each line of $S$ lies on at most $q+1$ planes of $S$, then $\gamma\le q+1$ and hence $3\gamma-2\le 3q-1\le q^2+1$, that is the coefficient of $x$ on the right hand side is larger than the one on the left hand side. Since $x\le (q+1)^2$, it follows that
\[
(3f(\beta)-2)((q+1)^2+q^3-q-2\beta)\le (q^2+1)(q+1)^2-\beta(q+1)-7\beta+3\theta_3.
\]
Multiplying this with the denominator of $f(\beta)$ gives $g(\beta)\ge 0$ for the polynomial
\begin{eqnarray*}
g(y)&:=&(q+1)[(6q^2-17q)y^2+(-3q^5+4q^4-8q^3-8q^2-5q-12)y
\\
&&+q^7+8q^6+15q^5+22q^4+27q^3+20q^2+13q+6].
\end{eqnarray*}
This is a polynomial of degree two with
\begin{eqnarray*}
f((q+1)^2)&=&-(q+1)^3(2q^5-16q^4+11q^3+24q^2+21q+6)
\end{eqnarray*}
and
\begin{eqnarray*}
f(\frac{q^2+5}2)=-\frac14(q+1)(2q^7-46q^6+3q^5-172q^4+152q^3-126q^2+423q+96).
\end{eqnarray*}
Since $q\ge 25$ we see that $f(\frac{q^2+5}2)<0$ and $f((q+1)^2)<0$ and hence $f(y)<0$ for all $y$ with
$\frac{q^2+5}2\le y\le (q+1)^2$. As $\frac{q^2+5}2\le \beta\le (q+1)^2$ and $g(\beta)\ge 0$, this is a contradiction.
\end{proof}

\begin{Le}
If $q\ge 25$, then $\beta=q+1$ and $x=(q+1)^2$. Also every solid and every line of $\PG(5,q)$ is incident with either $0$ or $q+1$ planes of $\LL$
\end{Le}
\begin{proof}
From Lemma \ref{betasmallorlarge} and Lemma \ref{betanottoolarge} we see that $\beta= q+1$. Then Lemma \ref{atleastqplus1}(b) shows that $x=(q+1)^2$ and that every line and solid is incident with either 0 or $q+1$ planes of $\LL$.
\end{proof}

The previous Lemma describes the situation of (b) in Theorem \ref{THW5q}. This finishes the proof of Theorem \ref{THW5q}.

\begin{remark}
Suppose that $x=(q+1)^2$ and that every line and solid is incident with at most $x$ planes of $\LL$. We have seen that this implies that every line and every solid is incident with either $0$ or $q+1$ planes of $\LL$. Suppose for every solid with $q+1$ planes of $\LL$ that these $q+1$ planes contain a common line. Then we will show that $\LL$ is the  set of planes of an embedded symplectic polar space $W(5,q)$. We were not able to prove this without the extra condition on how the planes contained in a solid intersect. We present this result as a theorem below.
\end{remark}

\begin{Th}
Suppose that $x=(q+1)^2$ and that every line and solid is incident with at most $x$ planes of $\LL$. Suppose for every solid with $q+1$ planes of $\LL$ that these $q+1$ planes contain a common line. Then  $\LL$ is the  set of planes of an embedded symplectic polar space $W(5,q)$.
\end{Th}
\begin{proof}
First note that the condition on $x$ implies that every line and every solid is incident with $0$ or $q+1$ elements of $\LL$.

Now consider the point-line geometry $\WW$ with point set the points of PG$(5,q)$ and line set the lines that are contained in a plane of $\LL$.

We will show that $\WW$  satisfies the one-or-all axiom. Let $P$ be any point and $\ell$ any line of $\WW$ not through $P$. If $P$ and $\ell$ are incident with a common element of $\LL$, then obviously $P$ is collinear with all points of $\ell$ in $\WW$. So suppose that no such element of $\LL$ exists. Consider any plane $\pi$ of $\LL$ through $\ell$, then the solid spanned by $\pi$ and $P$ must contain exactly $q+1$ planes of $\LL$ on a common line $\ell'$ distinct from $\ell$. Exactly one of these planes will contain $P$, and hence gives rise to a line of $\WW$ through $P$ intersecting $\ell$. So there is at least one point on $\ell$ collinear with $P$. Since $\ell$ is contained in a plane of $\LL$, it is contained in exactly $q+1$ planes of $\LL$. The solid spanned by any two of these planes then must contain exactly $q+1$ planes on a common line, which must be $\ell$. Hence the $q+1$ planes on $\ell$ are contained in a solid. Denote by $\ell^\perp$ the unique solid containing the $q+1$ planes of $\LL$ on $\ell$. Because every point is contained in $\theta_3$ planes of $\LL$, and every line of $\WW$ is contained in $q+1$ planes of $\LL$, we see that every point is incident with $\theta_3$ lines of $\WW$. Hence exactly $(q+1)(\theta_3-\theta_2)$ lines of $\WW$ not contained in $\ell^\perp$ intersect $\ell$ in a point. These lines cover at most $q (q+1)(\theta_3-\theta_2)=\theta_5-\theta_3$ points not contained in $\ell^\perp$. However, by the previous argument, each of the $\theta_5-\theta_3$ points not in $\ell^\perp$ are contained in at least one of these lines. Hence the point $P$ is on a unique line of $\WW$ that intersects $\ell$. We conclude that $\WW$ satisfies the one-or-all axiom. It now easily follows that $\WW$ is an embedded symplectic polar space $W(5,q)$, and that $\LL$ is the set of planes of this polar space.
\end{proof}

\section{The case $n=7$}\label{Sectionnis7}

In this section $\LL$ denotes an intriguing set of planes of $PG(7,q)$ corresponding to the second non-trivial eigenvalue $\lambda_2=q^5+q^4+q^3-q-1$ of the Grassmann graph. Hence $$|\LL|=x(q^2+1)(q^4+1)(q^2-q+1)$$ for some integer $x$. We shall show in this section that whenever $q\ge 13$ then $x\ge (q+1)(q^2+q+1)$ with equality only if $\LL$ is based on a line spread as in Example \ref{ex2}.

Recall that every plane of $\LL$ meets $\lambda_2+x$ planes of $\LL$ in a line and that every plane not in $\LL$ meets $x$ planes of $\LL$ in a line. We call a line of $\PG(7,q)$ a \textsl{full line}, if all planes on that line belong to $\LL$. Throughout we assume that $$x\le (q+1)(q^2+q+1).$$

\begin{Le}\label{fulllinesareskew}
(a) The full lines are mutually skew.

(b) A line that is not a full line, lies on at most $x$ planes of $\LL$.
\end{Le}
\begin{proof}
(a) As each full line lies in $\theta_5$ planes of $\LL$, a plane with two full lines would meet at least $2(\theta_5-1)$ planes of $\LL$ in one of these two lines. Since $2(\theta_5-1)>\lambda_2+x$, this is impossible. Hence, no plane contains two full lines.

(b) If $\pi$ is a plane on a line $\ell$ with $\pi\notin \LL$, then $\pi$ meets exactly $x$ planes of $\LL$ in a line and hence $\ell$ can be in at most this many planes of $\LL$.
\end{proof}

{\bf Notation}. If a line $\ell$ is not a full line, then we denote by $x-\delta_\ell$ the number of planes of $\LL$ that contain $\ell$.

\begin{Le}\label{deltasum}
Let $\pi$ be a plane of $\LL$ and suppose that no line of $\pi$ is a full line. Then
\[
\sum_g\delta_g=(q^2+q)x-q^2\theta_3\le q(q+1)(2q^2+q+1)
\]
where the sum is over all lines $g$ of $\pi$. In particular $x\ge q^3+q$.
\end{Le}
\begin{proof}
Since $\pi$ meets $\lambda_2+x$ planes of $\LL$ in a line, we have $\lambda_2+x=\sum (x-1-\delta_g)$, where the sum is over the lines $g$ of $\pi$.
\end{proof}

\begin{Le}\label{graph}
Suppose that $(S\cup T,E)$ is a bipartite graph with $|S|=|T|=q$ and $c$ is its matching number. Then the graph has at most $cq$ edges.
\end{Le}
\begin{proof}
Let $M$ be a matching of the graph with $c$ edges and put $M=\{\{v_i,w_i\}\mid i=1,\dots,c\}$ with $v_i\in S$ and $w_i\in T$. Since this matching is maximal, every edge of the graph must contain one of the vertices $v_i$ or one of the vertices $w_i$. Consider $i\le c$. There do not exist two edges $\{v_i,w\}$ and $\{v,w_i\}$ with $v\notin V:=\{v_1,\dots,v_c\}$ and $w\notin W:=\{w_1,\dots,w_c\}$, since otherwise $M\setminus\{\{v_i,w_i\}\}$ together with these two edges would be a matching with $c+1$ edges. Hence, there exist at most $q-c$ edges that contain a vertex of $\{v_i,w_i\}$ and a vertex not in $V\cup W$. This shows that the graph has at most $c^2+c(q-c)=cq$ edges.
\end{proof}

\begin{Le}\label{twopoorgivemany}
Suppose $q\ge 4$. If two different planes $\pi_1$ and $\pi_2$ of $\LL$ meet in a line and if neither $\pi_1$ nor $\pi_2$ contains a full line, then the solid they span contains at  least $q^2(q+1)-(2q^2+2q+5)$ planes of $\LL$.
\end{Le}
\begin{proof}
The planes $\pi_1$ and $\pi_2$ meet in a line $\ell$ and span a solid $S$. Let $s$ be the number of planes of $\LL$ that are contained in $S$. Let $M$ be the set consisting of all planes of $S$ that do not contain the line $\ell$.

Let $P$ be a point of $\ell$. Let $c$ be the largest number such that there exist planes $\tau_1,\dots,\tau_c\in M\setminus\LL$ on $P$ such $\tau_i\cap\pi_1\not=\tau_j\cap\pi_1$ and $\tau_i\cap\pi_2\not=\tau_j\cap\pi_2$ for all $i,j$ with $i\not=j$. Lemma \ref{graph} shows that the number of planes of $M\setminus\LL$ on $P$ is at most $qc$. Hence $P$ lies in at least $q(q-c)$ planes of $M\cap\LL$.

Suppose that $c\ge 1$ and consider an index $i$ with $1\le i\le c$. The plane $\tau_i$ meets exactly $x$ planes of $\LL$ in a line and $s$ of these planes lie in $S$. On the other hand, the line $g_1:=\tau_i\cap \pi_1$ lies on $x-\delta_{g_1}$ planes of $\LL$ and at most $q$ of these lie in $S$. The same holds for the line $g_2:=\tau_i\cap\pi_2$. Hence the plane $\tau_i$ meets at least $s+(x-q-\delta_{g_1})+(x-q-\delta_{g_2})$ planes of $\LL$ in a line. It follows that
\[
s+(x-q-\delta_{g_1})+(x-q-\delta_{g_2})\le x\ \Rightarrow\ \delta_{g_1}+\delta_{g_2}\ge s+x-2q.
\]
If $P_0,\dots,P_q$ are the points of $\ell$, and if we define $c_i$ for $P_i$ as we defined $c$ for $P$, then we have shown that at least $\sum_{i=0}^qq(q-c_i)$ of the $q^3+q^2$ planes of $M$ are planes of $\LL$ and that the sum $\sum\delta_g$ over all lines $g$ of $\pi_1$ and of $\pi_2$ with $g\not=\ell$ is at least $\sum_{i=0}^qc_i(s+x-2q)$. Lemma \ref{deltasum} therefore gives
\[
\sum_{i=0}^qc_i(s+x-2q)\le 2(\theta_2-1)x-2q^2\theta_3.
\]
We also know that
\[
s\ge \sum_{i=0}^qq(q-c_i)=(q+1)q^2-q\sum_{i=0}^qc_i.
\]
If we put $s=q^2(q+1)-y$, then $q\sum c_i\ge y$ and hence
\[
y(q^2(q+1)-y+x-2q)\le 2q(\theta_2-1)x-2q^3\theta_3.
\]
As $y\le q^2(q+1)$, this remains true, if we replace $x$ by its upper bound $(q+1)\theta_2$. Simplifying the resulting expression gives
\[
y(2q^3+3q^2+1-y)\le 2q^2(q+1)(2q^2+q+1).
\]
For $q\ge 4$, it follows easily that $y<2q^2+2q+5$.\end{proof}

\begin{Le}\label{deltag}
Let $S$ be a solid and let $s$ be the number of planes of $\LL$ that are contained in $S$.
\begin{enumerate}
\renewcommand{\labelenumi}{\rm(\alph{enumi})}
\item If $g$ is a line of $S$ such that some plane of $S$ on $g$ does not belong to $\LL$, then $\delta_g\ge s-q$.
\item If $g$ is a line of $S$ that is not full but meets a full line of $S$, then $\delta_g\ge s+q^2$.
\end{enumerate}
\end{Le}
\begin{proof}
(a) Let $\pi$ be a plane of $S$ on $g$ with $\pi\notin \LL$. Let $c$ be the number of planes of $\LL$ that lie in $S$ and contain $g$. The plane $\pi$ meets $x$ planes of $\LL$ in a line. Of these, exactly $s$ lie in $S$. Hence, $g$ lies on at most $x-s+c\le x-s+q$ planes of $\LL$. Therefore $\delta_g\ge s-q$.

(b) Let $h$ be a full line of $S$ meeting $g$ and let $\pi$ be the plane spanned by $g$ and $h$. Let $c$ be the number of planes of $\LL$ that lie in $S$ and contain $g$. As $\pi$ contains a full line, then $\pi\in\LL$ and hence $\pi$ meets $\lambda_2+x$ planes of $\LL$ in a line. Exactly $q^5+q^4+q^3+q^2$ of these meet $S$ in the full line $h$. Another $s-1$ of these lie in $S$. Also $x-\delta_g$ planes of $\LL$ contain $g$ and $c$ of these are contained in $S$. Hence $q^5+q^4+q^3+q^2+s+(x-\delta_g-c)\le \lambda_2+x$. Since $c\le q+1$, this gives $\delta_g\ge s+q^2$.
\end{proof}

\begin{Le}\label{poorandthreefull}
If a solid contains at least three full lines and a plane of $\LL$ without a full line, then the solid contains at most $\frac13(2q^3+2q+1)$ planes of $\LL$.
\end{Le}
\begin{proof}
Let $s$ be the number of planes of $\LL$ that lie in $S$. Consider a plane $\pi\in \LL$ that lies in $S$ and has no full line. As full lines are mutually disjoint (Lemma \ref{fulllinesareskew}), the three full lines meet $\pi$ in three different points and thus $\pi$ contains at least $3q$ lines $g$ containing one of these points. For these lines $g$ we have $\delta_g\ge s+q^2$ by Lemma \ref{deltag}. Therefore Lemma \ref{deltasum} shows that $3q(s+q^2)\le q(q+1)(2q^2+q+1)$. Solving for $s$ proves the assertion.
\end{proof}

\begin{Le}\label{richsolid}
Suppose that $q\ge 4$. Let $S$ be a solid, let $M$ be the set consisting of all planes of $S$ that lie in $\LL$ but do not contain a full line. If $|M|\ge 2$, then $|M|\ge q^3-q^2-4q-7$.
\end{Le}
\begin{proof}
If $|M|\ge 2$, then Lemma \ref{twopoorgivemany} shows that $S$ contains at least $q^2(q+1)-(2q^2+2q+5)$ planes from $\LL$. Then Lemma \ref{poorandthreefull} shows that $S$ contains at most two full lines, and hence $S$ has at most $2(q+1)$ planes that lie in $\LL$ and have a full line. Thus $|M|\ge q^2(q+1)-(2q^2+2q+5)-2(q+1)$.
\end{proof}

{\bf Notation}: We call a solid a rich solid, if all but at most $2q^2+5q+8$ planes of $S$ are planes of $\LL$ without a full line. A $4$-space is called a rich $4$-space, if all but at most $4q^3+10q^2+17q-1$ of its solids are rich. A $5$-space is called a rich $5$-space, if all but at most $8q^4+20q^3+34q^2-q-1$ of its $4$-spaces are rich.

\begin{Le}\label{richfourspace}
Suppose that $q\ge 4$ and that $V$ is a $4$-space that contains at least two rich solids. Then $V$ is a rich $4$-space.
\end{Le}
\begin{proof}
Let $S_1$ and $S_2$ be rich solids of $V$ and $\pi$ the plane in which they meet. There are $\theta_2q^2$ pairs $(\tau_1,\tau_2)$ with planes $\tau_1\not=\pi$ of $S_1$ and $\tau_2\not=\pi$ of $S_2$ such that the lines $\pi_1\cap \pi$ and $\pi_2\cap\pi$ are the same. Each of the $q^4+q^3+q^2$ solids of $V$ not containing $\pi$ is spanned by the two planes of such a pair. Also, if the two planes $\tau_1$ and $\tau_2$ are planes of $\LL$ without a full line, then the solid they span is a rich solid (Lemma \ref{richsolid}). Since all but at most $2q^2+5q+8$ planes and as many planes of $S_2$ do not have this property and each occurs in exactly $q$ pairs, it follows that at most $2(2q^2+5q+8)q$ of the solids of $V$ that do not contain $\pi$ are not rich. From the $q+1$ solids of $V$ that contain $\pi$ at least $S_1$ and $S_2$ are rich, so at most $q-1$ are not rich. Hence $V$ has at most $4q^3+10q^2+17q-1$ solids that are not rich.
\end{proof}

\begin{Le}\label{richfivespace}
Suppose that $q\ge 4$ and that $F$ is a $5$-space that contains at least two rich $4$-spaces. Then $F$ is a rich $5$-space.
\end{Le}
\begin{proof}
This is a very similar argument. Consider two rich $4$-solids of $F$ and their intersection $S$. Then $q^5+q^4+q^3+q^2$ $4$-spaces of $F$ do not contain $S$ and at most $2(4q^3+10q^2+17q-1)q$ of these are not rich. Since $S$ lies in $q+1$ $4$-spaces of $F$ and two of these are rich, this proves the assertion.
\end{proof}

\begin{Le}\label{norich5space}
Suppose $q\ge 13$. Then there does not exist a rich $5$-space.
\end{Le}
\begin{proof}
Suppose that $F$ is a rich $5$-space. We count triples $(\pi,S,V)$ of planes $\pi\in\LL$, rich solids $S$ and rich $4$-spaces $V$ with $\pi\subset S\subset V\subset F$ in two ways. By choosing first $V$, then $S$ and then $\pi$, Lemma \ref{richfivespace}, \ref{richfourspace} and \ref{richsolid} show that the number of these triples is at least
\[
(\theta_5-(8q^4+20q^3+34q^2-q-1)(\theta_4-(4q^3+10q^2+17q-1))(\theta_3-(2q^2+5q+8)).
\]
Lemma \ref{incidenceswithpointsandhyperplanes} shows that each hyperplane contains exactly $\frac{x\theta_4\theta_5}{(q+1)^2\theta_2}$ planes of $\LL$. Hence $F$ contains at most $\frac{x\theta_4\theta_5}{(q+1)^2\theta_2}$ planes of $\LL$. A plane of $F$ lies in $\theta_2$ solids of $F$ and a solid of $F$ lies in $q+1$ $4$-spaces of $F$. As $x\le (q+1)\theta_2$, it follows that the number of triples we are counting is at most $\theta_2\theta_4\theta_5$. Comparing the lower and upper bound for the number of triples gives
a contradiction for $q\ge 11$.
\end{proof}

We finally show that the absence of rich $5$-spaces implies that every plane of $\LL$ contains a full line.

\begin{Le}
If $q\ge 13$, then every plane of $\LL$ contains a full line.
\end{Le}
\begin{proof}
Assume on the contrary that there exists a plane $\pi\in\LL$ that does not contain a full line. Then every line $g$ of $\pi$ lies in $x-d_g$ planes of $\LL$ with $d_g\ge 0$ and Lemma \ref{deltasum} shows that $x\ge q^3+q$ and $\sum d_g\le (q^2+q)x-q^2\theta_3\le \theta_2(x-q^3-q)$. Hence there exists a line $g$ in $\pi$ with $d_g\le x-q^3-q$, that is $g$ lies on at least $q^3+q$ planes of $\LL$.

If $\tau$ is a plane of $\LL$ on $g$ that contains a full line $h$, then $h$ meets $g$ and $\tau$ is spanned by $g$ and $h$. Since the full lines are mutually skew, we see that at most $q+1$ of the planes on $g$ contain a full line. Hence $g$ lies on at least $q^3-1$ planes of $\LL$ that do not contain a full line. Since $q\ge 13$, then $q^3-1>q^2+q+1$ and hence these planes can not all be contained in a $4$-space on $g$, since such a $4$-space contains only $q^2+q+1$ planes on $g$.

Hence we find four planes $\pi_1,\pi_2,\pi_3,\pi_4\in\LL$ on $g$ that do not contain a full line and such that these four planes span a $5$-space $F$. It follows from Lemma \ref{richsolid} that any two of these planes span a rich solid. Therefore Lemma \ref{richfourspace} shows that any three of these four planes span a rich $4$-space. Then Lemma \ref{richfivespace} shows that $F$ is a rich $5$-space. This contradicts Lemma \ref{norich5space}.
\end{proof}

Lemma \ref{linespread} now finishes the proof of Theorem \ref{Main}

\section{Final Comments}

In this paper we constructed the first infinite family of intriguing sets that are not tight, and as the results indicate, the extremal examples are related to interesting geometric structures such as the symplectic polar space $W(5,q)$ and line spreads. However an abundance of questions remain:

\begin{itemize}
\item Are there examples of intriguing sets that are not tight in the Grassmann graph of planes of PG$(n,q)$ when $n$ is even?
\item Are there  examples of intriguing sets that are not tight in the Grassmann graph of $k$-spaces of PG$(n,q)$ when $k>2$?
\item More generally, what about infinite families of intriguing sets that are not tight in distance regular graphs?
\end{itemize}

\begin{section}*{Acknowledgement}
This material is based upon work that was done while Stefaan De Winter was serving at the National Science Foundation. Stefaan De Winter also thanks the Mathematisches Institut at Justus-Liebig-Universit\"{a}t for its warm hospitality during his research visit there.
\end{section}

\bibliographystyle{plain}

\bibliography{literatur}

\end{document}